\documentclass{amsart}

\usepackage{amsfonts, amsmath, amssymb, amsthm, color,graphicx}

\allowdisplaybreaks[1]

\makeatletter
\numberwithin{equation}{section}
\numberwithin{figure}{section}
\theoremstyle{plain}
\newtheorem{thm}{\protect\theoremname}
  \theoremstyle{plain}
  \newtheorem{lem}[thm]{\protect\lemmaname}
  \theoremstyle{remark}
  
  \theoremstyle{plain}
  \newtheorem{prop}[thm]{\protect\propositionname}

\makeatother

  \providecommand{\lemmaname}{Lemma}
  \providecommand{\propositionname}{Proposition}
  \providecommand{\remarkname}{Remark}
\providecommand{\theoremname}{Theorem}

\begin{document}

\title[On Yamabe type problems on  Riemannian manifolds
with boundary]{On Yamabe type problems on  Riemannian manifolds
with boundary}

\author{Marco Ghimenti}
\address[Marco Ghimenti]{Dipartimento di Matematica
Universit\'a di Pisa
Via F. Buonarroti 1/c , I - 56127 Pisa, Italy}
\email{marco.ghimenti@dma.unipi.it }
\author{Anna Maria Micheletti}
\address[Anna Maria Micheletti]{Dipartimento di Matematica
Universit\'a di Pisa
Via F. Buonarroti 1/c , I - 56127 Pisa, Italy}
\email{a.micheletti@dma.unipi.it }

\author{Angela Pistoia}
\address[Angela Pistoia] {Dipartimento SBAI, Universt\`{a} di Roma ``La Sapienza", via Antonio Scarpa 16, 00161 Roma, Italy}
\email{angela.pistoia@uniroma1.it}

\begin{abstract}
Let $(M,g)$ be a $n-$dimensional compact Riemannian manifold with boundary. We consider the Yamabe type problem  
\begin{equation} 
\left\{ \begin{array}{ll}
-\Delta_{g}u+au=0 & \text{ on }M\\
 \partial_\nu u+\frac{n-2}{2}bu=  u^{{n\over n-2}\pm\varepsilon}  & \text{ on }\partial M
\end{array}\right.
\end{equation}
where $a\in C^1(M),$ $b\in C^1(\partial M)$, $\nu$ is the outward pointing unit normal to $\partial M $ and $\varepsilon$ is a small positive parameter. 
We build  solutions  which blow-up at a point of the   boundary  as $\varepsilon$ goes to zero. The blowing-up behavior is ruled by    the function $b-H_g ,$ where $H_g$  is the boundary mean curvature.
\end{abstract}


\subjclass[2000]{}

\maketitle

\section{Introduction}
Let $(M,g)$ be a smooth, compact Riemannian manifold of dimension
$n\ge3$ with a boundary $\partial M$ which is the union of a finite
number of smooth closed compact submanifolds embedded in $M$.\\

A well known problem in differential geometry is whether $(M,g)$ is necessarily conformally equivalent to a manifold of constant scalar curvature whose   boundary is minimal.
When the boundary is empty this is called the Yamabe problem (see Yamabe \cite{Y}), which has been completely solved by Aubin [A], Schoen \cite{S} and Trudinger \cite{T}.
Cherrier \cite{C} and Escobar \cite{E1,E2} studied the problem in the context of manifolds with boundary and gave an affirmative solution to the question in almost every case. The remaining cases where studied  by Marques \cite{M1,M2}, by Almaraz \cite{A1} and  by Brendle and Chen \cite{BC}.
\\
 
Once the problem is solvable, a natural question about compactness of the full set of solutions arises. Concerning the Yamabe problem, it was first raised    by Schoen in a topic course at Stanford University in 1988. 
A necessary condition is that the manifold is not conformally equivalent to the standard sphere $\mathbb S^n,$ since the group of conformal transformation of the round sphere is not compact itself. The problem of compactness  has been widely studied in the last years and it has been completely solved by Brendle \cite{B}, Brendle and Marques \cite{BM} and Khuri, Marques and Schoen \cite{KMS}.

In the presence of a boundary, 
a necessary condition is that
$M$ is not conformally equivalent to the standard ball $\mathbb B^n.$ The problem when the boundary of the manifold is not empty
  has been studied    by   V.
Felli and M. Ould Ahmedou   \cite{FO1,FO2}, Han and Li \cite{HL} and by Almaraz \cite{A2,A3} . In particular,   Almaraz  studied the compactness property in the case of  scalar-flat metrics. Indeed the zero scalar curvature case is  particularly interesting because it leads to study  a linear equation in the interior with a critical Neumann-type nonlinear boundary condition
\begin{equation}\label{ym}
\left\{ \begin{array}{ll}
-\Delta_{g}u+{n-2\over4(n-1)}R_gu=0 & \text{ on }M,\quad
u>0  \text{ in }M\\
 \partial_\nu u+\frac{n-2}{2}H_gu=  u^{n\over n-2}  & \text{ on }\partial M
\end{array}\right.
\end{equation}
where $\nu$ is the outward pointing unit normal to $\partial M,$ $R_g$ is the scalar  curvature of $M$ with respect to $g$ and $H_g$ is the boundary mean curvature with respect to $g$.

We note that in this case compactness of solutions is equivalent to establish a priori estimates for solutions to equation \eqref{ym}.
Almaraz in \cite{A3} proved that  compactness holds for a generic metric $g$. On the other hand, in \cite{A2} proved that if the dimension of the manifold is $n\ge25$ compactness does not hold  because it is possible to build blowing-up solutions to problem \eqref{ym} for a suitable metric  $g.$
We point out that the problem of compactness when the dimension of the manifold $n\le 24$   is still   not completely understood.
\\

An interesting issue, closely related to compactness property, is the stability problem. 
 One can ask whether or not the compactness property is preserved under perturbations of the equation, which
is equivalent to have or not uniform a-priori estimates for solutions of the perturbed problem. Let us consider the more general problem 
\begin{equation}\label{gp}
\left\{ \begin{array}{ll}
-\Delta_{g}u+a(x)u=0 & \text{ in } M,\quad u>0 \text{ in } M\\
 \partial_\nu u+b(x)u=  u^{n\over n-2}  & \text{ on }\partial M.
\end{array}\right.
\end{equation}
We say that  problem \eqref{gp} is stable if for any sequences of $C^1$ functions $a_\varepsilon:M\to \mathbb R$ and $b_\varepsilon:\partial M\to \mathbb R$ converging in $C^1$ to functions $a :M\to \mathbb R$ and $b :\partial M\to \mathbb R,$ for any sequence of exponents $p_\varepsilon:={n\over n-2}\pm\varepsilon$
converging to the critical one ${n\over n-2}$ and for any sequence of associated  solutions $u_\epsilon$ bounded in $H_g^1(M)$ of the perturbed problems
\begin{equation}\label{pp}
\left\{ \begin{array}{ll}
-\Delta_{g}u+a_\varepsilon(x)u=0 & \text{ in } M,\quad u_\varepsilon>0 \text{ in } M\\
 \partial_\nu u+{n-2\over2 }b_\varepsilon(x)u=  u_\varepsilon^{{n\over n-2}\pm\varepsilon}  & \text{ on }\partial M.
\end{array}\right.
\end{equation}
there is a subsequence $u_{\varepsilon_k}$ that converges in $C^2$ to a solution to the limit problem \eqref{gp}.
The stability of the Yamabe problem has been introduced and studied by Druet in \cite{D1,D2} and by Druet and Hebey in \cite{DH1,DH2}.
Recently, Esposito Pistoia and Vetois \cite{EPV}, Micheletti, Pistoia and Vetois \cite{MPV} and Esposito and Pistoia \cite{EP}  prove that a priori estimates fail for perturbations of the linear potential or of the exponent.
\\

In the present paper, we investigate the question of stability of problem \eqref{gp}.
It is clear that it is not stable if it possible to build solutions $u_\varepsilon$ to perturbed problems \eqref{pp} which blow-up at one or more points of the manifold as the parameter $\varepsilon$ goes to zero. Here,  we show that the behavior of the sequence $u_\epsilon$ is dictated by the difference 
\begin{equation}\label{diff}
\varphi(q)=b(q)-H_{g}(q)\ \text{ for } q\in\partial M,
\end{equation}
 More precisely, we will  consider the problem 
\begin{equation}
\left\{ \begin{array}{ll}
-\Delta_{g}u+a(x)u=0 & \text{ on }M, \quad u >0 \text{ in } M\\
\frac{\partial}{\partial\nu}u+\frac{n-2}{2}b(x)u=  u^{{n\over n-2}\pm\varepsilon} \text{ on }\partial M
\end{array}\right.\label{eq:1.2}
\end{equation}
We will assume that   $a\in C^1(M)$, $b\in C^1(\partial M)$ are such that the linear operator ${\mathcal L}u:=-\Delta_gu+a u$ with Neumann boundary condition ${\mathcal B} u:=\partial _\nu u+{n-2\over 2} bu$ is coercive, namely there exists a constant $c>0$ so that
\begin{equation}\label{coe}
\int_{M}\left(|\nabla_{g}u|^{2}+a(x)u^{2}\right)d\mu_{g}+\frac{n-2}{2}\int_{\partial M}b(x)u^{2}d\sigma\ge c\|u\|_{H ^{1}(M)}^{2}.
\end{equation}
Here $\varepsilon$ is a small positive parameter. The problem \eqref{eq:1.2} turns out  to be either slightly subcritical or slightly supercritical  if the exponent in the nonlinearity is either ${{n\over n-2}-\varepsilon}$ or ${{n\over n-2}+\varepsilon}$, respectively.
Let us state our main result.

\begin{thm}
\label{thm:main} Assume \eqref{coe} and $n\ge7$. 
\begin{itemize}
\item[(i)]
If $q_{0}\in\partial M$ is a strict local minimum point of the function $\varphi$ defined in \eqref{diff} with $\varphi(q_{0})>0,$
then provided  $\varepsilon>0$ is small enough there exists
a solution $u_{\varepsilon}$ of (\ref{eq:1.2}) in the slightly subcritical
case such that $u_{\varepsilon}$ blows up at a boundary point when $\varepsilon\rightarrow0^{+}$.

\item[(ii)]
If $q_{0}\in\partial M$ is a strict local maximum point of the function $\varphi$ defined in \eqref{diff} with $\varphi(q_{0})>0,$
then provided  $\varepsilon>0$ is small enough there exists
a solution $u_{\varepsilon}$ of (\ref{eq:1.2}) in the supercritical
case such that $u_{\varepsilon}$ blows up at a boundary point when $\varepsilon\rightarrow0^{+}$.
\end{itemize}
\end{thm}
 
 Our result does not concern  the stability of the geometric Yamabe problem \eqref{ym}. Indeed,  the function $\varphi$ in \eqref{diff} turns out to be identically zero.
 In this case it is interesting to discover  the function which rules the behavior of blowing-up sequences in this case.
 We expect that    it    depends on  trace-free 2nd fundamental form  as it is suggested by Almaraz in \cite{A3}, where a compactness result in the subcritical case is established.\\

 It also remains open the case of low dimension, where we expect that the function $\varphi$ in \eqref{diff}
should be replaced by a function which depends on the Weyl tensor of the boundary, as it is suggested by Escobar in \cite{E1,E2}.
\\

The proof of our result relies on a very well known Ljapunov-Schmidt procedure. In Section \ref{uno} we set the problem, in Section \ref{due} we reduce the problem to a finite dimensional one, which is studied in  Section \ref{tre}.
\section{Setting of the problem }\label{uno}

Let us rewrite problem \eqref{eq:1.2} in a more convenient way.

First of all, assumption \eqref{coe} allows to   endow the Hilbert space $H:=H^{1} (M)$ with the following scalar
product 
\[
\left\langle \left\langle u,v\right\rangle \right\rangle _{H}:=\int_{M}\left(\nabla_{g}u\nabla_{g}v+a(x)uv\right)d\mu_{g}+\frac{n-2}{2}\int_{\partial M}b(x)uvd\sigma
\]
and the induced norm $\|u\|_{H}^{2}:=\left\langle \left\langle u,u\right\rangle \right\rangle _{H}.$
We define the exponent 
\[
s_{\varepsilon}=\left\{ \begin{array}{cc}
\frac{2(n-1)}{n-2} & \text{in the subcritical case}\\
\\
\frac{2(n-1)}{n-2}+n\varepsilon & \text{in the supercritical case}
\end{array}\right.
\]
and the Banach space $\mathcal{H}:=H^{1}(M)\cap L^{s_{\varepsilon}}(\partial M)$
endowed with norm $\|u\|_{\mathcal{H}}=\|u\|_{H}+|u|_{L^{s_{\varepsilon}}(\partial M)}.$

We notice that in the subcritical case $\mathcal{H}$ is nothing but   
the Hilbert space $H$. 

By trace theorems, we have the following inclusion $W^{1,\tau}(M)\subset L^{t}(\partial M)$
for $t\le\tau\frac{n-1}{n-\tau}$. 

We consider $i:H^{1}(M)\rightarrow L^{\frac{2(n-1)}{n-2}}(\partial M)$
and its adjoint with respect to $\left\langle \left\langle \cdot,\cdot\right\rangle \right\rangle _{H}$
\[
i^{*}:L^{\frac{2(n-1)}{n}}(\partial M)\rightarrow H^{1}(M)
\]
defined by
\[
\left\langle \left\langle \varphi,i^{*}(g)\right\rangle \right\rangle _{H}=\int_{\partial M}\varphi gd\sigma\text{ for all }\varphi\in H^{1}
\]
so that $u=i^{*}(g)$ is the weak solution of the problem
\begin{equation}
\left\{ \begin{array}{ll}
-\Delta_{g}u+a(x)u=0 & \text{ on }M\\
\frac{\partial}{\partial\nu}u+\frac{n-2}{2}b(x)u=g & \text{ on }\partial M
\end{array}\right..\label{eq:istellasopra}
\end{equation}

We recall that by (see \cite{N}) we have that, if $u\in H^{1}$ is a
solution of (\ref{eq:istellasopra}), then for $\frac{2n}{n+2}\le q\le\frac{n}{2}$
and $r>0$ it holds
\begin{equation}
\|u\|_{L^{\frac{(n-1)q}{n-2q}}(\partial M)}=\|i^{*}(g)\|_{L^{\frac{(n-1)q}{n-2q}}(\partial M)}\le\|g\|_{L^{\frac{(n-1)q}{n-q}+r}(\partial M)}.\label{eq:nittka}
\end{equation}
By this result, we can choose $q,r$ such that
\begin{equation}
\frac{(n-1)q}{n-2q}=\frac{2(n-1)}{n-2}+n\varepsilon\text{ and }\frac{(n-1)q}{n-q}+r=\frac{2(n-1)+n(n-2)\varepsilon}{n+(n-2)\varepsilon}\label{eq:nittka1}
\end{equation}
that is 
\[
q=\frac{2n+n^{2}\left(\frac{n-2}{n-1}\right)\varepsilon}{n+2+2n\left(\frac{n-2}{n-1}\right)\varepsilon}\text{ and }r=\frac{2(n-1)+n(n-2)\varepsilon}{n+(n-2)\varepsilon}-\frac{2(n-1)+n(n-2)\varepsilon}{n+\left(n-2\right)\left(\frac{n}{n-1}\right)\varepsilon};
\]
so we have that, if $u\in L^{\frac{2(n-1)}{n-2}+n\varepsilon}(\partial M)$,
then $ |u|^{{n\over n-2}+\varepsilon}\in L^{\frac{2(n-1)+n(n-2)\varepsilon}{n+\varepsilon(n-2)}}(\partial M)$
and, in light of (\ref{eq:nittka}), that also $i^{*}\left( |u|^{{n\over n-2}+\varepsilon}\right)\in L^{\frac{2(n-1)}{n-2}+n\varepsilon}(\partial M)$. 

Finally, we rewrite  problem (\ref{eq:1.2}) -both in the subcritical
and in the supercritcal case- as 
\begin{equation}\label{pre}
u=i^{*}\left(f_{\varepsilon}(u)\right),\ u\in\mathcal{H},
\end{equation}
where the nonlinearity $f_{\varepsilon}(u)$ is defined as  $f_{\varepsilon}(u)  :=(n-2)(u^+)^{\frac{n}{n-2}+\varepsilon}$
in  the supercritical case or  $f_{\varepsilon}(u) :=(n-2)(u^+)^{\frac{n}{n-2}-\varepsilon}$
in the subcritical case. Here $u^+(x):=\max\{0,u(x)\}$.
By assumption \eqref{coe},  a solution to problem \eqref{pre} is strictly positive and actually it is a solution to problem \eqref{eq:1.2}.
Therefore, we are led to build   solutions to problem \eqref{pre} which blow-up at a boundary point as $\varepsilon$ goes to zero.

The main ingredient to cook up our solutions are   the standard bubbles
$$
U_{\delta,\xi}(x,t):=\frac{\delta^{\frac{n-2}{2}}}{\left((\delta+t)^{2}+|x-\xi|^{2}\right)^{\frac{n-2}{2}}}, \ (x,t)\in\mathbb{R}^{n-1}\times\mathbb{R}_{+},\ \delta>0,\ \xi\in\mathbb{R}^{n-1},$$
which are all the solutions to the limit problem
\begin{equation}
\left\{ \begin{array}{ll}
-\Delta U=0 & \text{ on }\mathbb{R}^{n-1}\times\mathbb{R}_{+}\\
 \partial_\nu U=(n-2)U^{\frac{n}{n-2}} & \text{ on }\mathbb{R}^{n-1}\times\left\{ t=0\right\} .
\end{array}\right.\label{eq:PL}
\end{equation}
We set $U_{\delta}(x,t) :=U_{\delta,0}(x,t).$

We also need to introduce the  linear  problem  
\begin{equation}
\left\{ \begin{array}{ll}
-\Delta V=0 & \text{ on }\mathbb{R}^{n-1}\times\mathbb{R}_{+}\\
\partial_\nu V=nU_{1}^{\frac{2}{n-2}}V & \text{ on }\mathbb{R}^{n-1}\times\left\{ t=0\right\} .
\end{array}\right.\label{eq:linearizzato}
\end{equation}
In \cite{A3} it has been proved  that  the $n-$dimensional space of  solutions of (\ref{eq:linearizzato}) is generated by the functions
\[
V_{i}=\frac{\partial U_{1}}{\partial x_{i}}=(2-n)\frac{x_{i}}{\left((1+t)^{2}+|x|^{2}\right)^{\frac{n}{2}}}\text{ for }i=i,\dots n-1
\]
\[
V_{0}=\left.\frac{\partial U_{\delta}}{\partial\delta}\right|_{\delta=1}=\frac{n-2}{2}\left(\frac{1}{(1+t)^{2}+|x|^{2}}\right)^{\frac{n}{2}}\left[t^{2}+|x|^{2}-1\right]
\]

Next, given a point $q\in\partial M,$ we introduce the Fermi coordinates $\psi_{q}^{\partial}:B^{n-1}(0,R)\times[0,R)\rightarrow M$,
where $B^{n-1}(0,R)$ is the $n-1$ dimensional unitary ball in $\mathbb{R}^{n-1}$ and
we   read   the bubble on the manifold as the function
\[
W_{\delta,q}(\xi)=U_{\delta}\left((\psi_{q}^{\partial})^{-1}\xi\right)\chi\left((\psi_{q}^{\partial})^{-1}\xi\right),
\]
and the functions $V_i$'s on the manifold as the functions
\[
Z_{\delta,q}^{i}(\xi)=\frac{1}{\delta^{\frac{n-2}{2}}}V_{i}\left(\frac{1}{\delta}(\psi_{q}^{\partial})^{-1}\xi\right)\chi\left((\psi_{q}^{\partial})^{-1}\xi\right)\ \ i=0,\dots n-1.
\]
where 
$\chi(x,t)=\tilde{\chi}(|x|)\tilde{\chi}(t)$, being $\tilde{\chi}$
a smooth cut off function, $\tilde{\chi}(s)\equiv1$ for $0\le s<R/2$
and $\tilde{\chi}(s)\equiv0$ for $s\ge R$. 
Then, it is necessary to split the Hilbert space
$H$  into the sum of the orthogonal spaces
$$
K_{\delta,q}=\text{Span }\left\langle Z_{\delta,q}^{0},\dots,Z_{\delta,q}^{n-1}\right\rangle 
$$ and
$$
K_{\delta,q}^{\bot}=\left\{ \varphi\in H^{1}(M)\ |\ \left\langle \left\langle \varphi,Z_{\delta,q}^{i}\right\rangle \right\rangle _{H}=0\ \text{ for all }i=0,\dots,n-1\right\}. 
$$
Finally, we can look for a solution to problem \eqref{pre} as
\[
u_\varepsilon(x)=W_{\delta,q}(x)+\phi (x)
\]
where the blow-up point $q\in\partial M,$ the blowing-up rate $\delta$ satisfies
\begin{equation}\label{delta}
\delta:=d\varepsilon\ \hbox{for some}\ d>0
\end{equation}
and the remainder term $\phi $ belongs to the infinite dimensional space $K_{\delta,q}^{\bot}\cap\mathcal{H} $ of codimension $n$.  
We are led to solve the system
\begin{equation}
\Pi_{\delta,q}^{\bot}\left\{ W_{\delta,q}(x)+\phi(x)-i^{*}\left(f_{\varepsilon}(W_{\delta,q}(x)+\phi(x))\right)\right\} =0\label{eq:rid1}
\end{equation}
\begin{equation}
\Pi_{\delta,q}\left\{ W_{\delta,q}(x)+\phi(x)-i^{*}\left(f_{\varepsilon}(W_{\delta,q}(x)+\phi(x))\right)\right\} =0\label{eq:rid2}
\end{equation}
being $\Pi_{\delta,q}^{\bot}$ and $\Pi_{\delta,q}$ the projection
respectively on $K_{\delta,q}^{\bot}$ and $K_{\delta,q}$.

\section{The finite dimensional reduction\label{sec:red} }\label{due}

In this section we perform the finite dimensional reduction. We rewrite
the auxiliary equation (\ref{eq:rid1}) in the equivalent form 
\begin{equation}
L(\phi)=N(\phi)+R\label{eq:rid1equiv}
\end{equation}
where $L=L_{\delta,q}:K_{\delta,q}^{\bot}\cap\mathcal{H}\rightarrow K_{\delta,q}^{\bot}\cap\mathcal{H}$
is the linear operator 
\[
L(\phi)=\Pi_{\delta,q}^{\bot}\left\{ \phi(x)-i^{*}\left(f_{\varepsilon}'(W_{\delta,q})[\phi]\right)\right\} ,
\]
$N(\phi)$ is the nonlinear term 
\begin{equation}
N(\phi)=\Pi_{\delta,q}^{\bot}\left\{ i^{*}\left(f_{\varepsilon}(W_{\delta,q}(x)+\phi(x))\right)-i^{*}\left(f_{\varepsilon}(W_{\delta,q}(x)\right)-i^{*}\left(f_{\varepsilon}'(W_{\delta,q})[\phi]\right)\right\} \label{eq:Ndef}
\end{equation}
and the error term $R$ is defined by
\begin{equation}
R=\Pi_{\delta,q}^{\bot}\left\{ i^{*}\left(f_{\varepsilon}(W_{\delta,q}(x)\right)-W_{\delta,q}(x)\right\} .\label{eq:Rdef}
\end{equation}

\subsection{The invertibility of the linear operator $L$}
\begin{lem}
\label{lem:L}  For $a,b\in\mathbb{R}$,
$0<a<b$ there exists a positive constant $C_{0}=C_{0}(a,b)$ such
that, for $\varepsilon$ small, for any $q\in\partial M$, for any
$       d    \in[a,b]$ and for any $\phi\in K_{\delta,q}^{\bot}\cap\mathcal{H}$
there holds
\[
\|L_{\delta,q}(\phi)\|_{\mathcal{H}}\ge C_{0}\|\phi\|_{\mathcal{H}}.
\]
\end{lem}
\begin{proof}
We argue by contradiction. We suppose that there exist two sequence
of real numbers $\varepsilon_{m}\rightarrow0,       d    _{m}\in[a,b]$
a sequence of points $q_{m}\in\partial M$ and a sequence of functions
$\phi_{\varepsilon_{m}       d    _{m},q_{m}}\in K_{\varepsilon_{m}       d    _{m},q_{m}}^{\bot}\cap\mathcal{H}$
such that 
\[
\|\phi_{\varepsilon_{m}       d    _{m},q_{m}}\|_{\mathcal{H}}=1\text{ and }\|L_{\varepsilon_{m}       d    _{m},q_{m}}(\phi_{\varepsilon_{m}       d    _{m},q_{m}})\|_{\mathcal{H}}\rightarrow0\text{ as }m\rightarrow+\infty.
\]
For the sake of simplicity, we set $\delta_{m}=\varepsilon_{m}       d    _{m}$
and we define
\[
\tilde{\phi}_{m}:=\delta_{m}^{\frac{n-2}{2}}\phi_{\delta_{m},q_{m}}(\psi_{q_{m}}^{\partial}(\delta_{m}\eta))\chi(\delta_{m}\eta)\text{ for }\eta=(z,t)\in\mathbb{R}_{+}^{n},\text{ with }z\in\mathbb{R}^{n-1}\text{ and }t\ge0
\]
Since $\|\phi_{\varepsilon_{m}       d    _{m},q_{m}}\|_{H}\le1$, by change
of variables we easily get that $\left\{ \tilde{\phi}_{m}\right\} _{m}$
is bounded in $D^{1,2}(\mathbb{R}_{+}^{n})$ (but not in $H^{1}(\mathbb{R}_{+}^{n})$).
Thus there exists $\tilde{\phi}\in D^{1,2}(\mathbb{R}_{+}^{n})$ such
that $\tilde{\phi}_{m}\rightharpoonup\tilde{\phi}$ weakly in $D^{1,2}(\mathbb{R}_{+}^{n})$,
in $L^{\frac{2n}{n-2}}(\mathbb{R}_{+}^{n})$, strongly in $L_{\text{loc}}^{\frac{2(n-1)}{n-2}}(\partial\mathbb{R}_{+}^{n})$
and almost everywhere.

Since $\phi_{\delta_{m},q_{m}}\in K_{\delta_{m},q_{m}}^{\bot}$, and
taking in account (\ref{eq:linearizzato}) we get, for $i=0,\dots,n-1$,
\begin{equation}
o(1)=\int_{\mathbb{R}_{+}^{n}}\nabla\tilde{\phi}\nabla V_{i}dzdt=n\int_{\mathbb{R}^{n-1}}U_{1}^{\frac{2}{n-2}}(z,0)V_{i}(z,0)\tilde{\phi}(z,0)dz.\label{eq:L3}
\end{equation}
 Indeed, by change of variables we have
\begin{align*}
0= & \left\langle \left\langle \phi_{\delta_{m},q_{m}},Z_{\delta_{m},q_{m}}^{i}\right\rangle \right\rangle _{H}=\int_{M}\left(\nabla_{g}\phi_{\delta_{m},q_{m}}\nabla_{g}Z_{\delta_{m},q_{m}}^{i}+a(x)\phi_{\delta_{m},q_{m}}Z_{\delta_{m},q_{m}}^{i}\right)d\mu_{g}\\
 & +\frac{n-2}{2}\int_{\partial M}b(x)\phi_{\delta_{m},q_{m}}Z_{\delta_{m},q_{m}}^{i}d\sigma\\
= & \int_{\mathbb{R}_{+}^{n}}|g_{q_{m}}(\delta\eta)|^{\frac{1}{2}}\delta^{\frac{n-2}{2}}g_{q_{m}}^{\alpha\beta}(\delta\eta)\frac{\partial}{\partial\eta_{\alpha}}V_{i}(\eta)\chi(\delta\eta)\frac{\partial}{\partial\eta_{\alpha}}\phi_{\delta_{m},q_{m}}(\psi_{q_{m}}^{\partial}(\delta_{m}\eta))d\eta\\
 & +\int_{\mathbb{R}_{+}^{n}}|g_{q_{m}}(\delta\eta)|^{\frac{1}{2}}\delta^{\frac{n+2}{2}}a(\psi_{q_{m}}^{\partial}(\delta\eta))V_{i}(\eta)\phi_{\delta_{m},q_{m}}(\psi_{q_{m}}^{\partial}(\delta_{m}\eta))d\eta\\
 & +\int_{\partial\mathbb{R}_{+}^{n}}|g_{q_{m}}(\delta z,0)|^{\frac{1}{2}}\delta^{\frac{n}{2}}b(\psi_{q_{m}}^{\partial}(\delta\eta))\phi_{\delta_{m},q_{m}}(\psi_{q_{m}}^{\partial}(\delta_{m}z,0))V_{i}(\delta_{m}z,0)dz\\
= & \int_{\mathbb{R}_{+}^{n}}\nabla V_{i}(\eta)\nabla\tilde{\phi}_{m}(\eta)+\delta^{2}a(q_{m})V_{i}(\eta)\tilde{\phi}_{m}(\eta)d\eta\\
 & +\delta\int_{\partial\mathbb{R}_{+}^{n}}b(q_{m})V_{i}(z,0)\tilde{\phi}_{m}(z,0)d\eta+O(\delta)=\int_{\mathbb{R}_{+}^{n}}\nabla V_{i}(\eta)\nabla\tilde{\phi}_{m}(\eta)+O(\delta)\\
= & \int_{\mathbb{R}_{+}^{n}}\nabla V_{i}(\eta)\nabla\tilde{\phi}(\eta)+o(1),
\end{align*}
By definition of $L_{\delta_{m},q_{m}}$ we have 
\begin{equation}
\phi_{\delta_{m},q_{m}}-i^{*}\left(f_{\varepsilon}'(W_{\delta_{m},q_{m}})[\phi_{\delta_{m},q_{m}}]\right)-L_{\delta_{m},q_{m}}\left(\phi_{\delta_{m},q_{m}}\right)=\sum_{i=0}^{n-1}c_{m}^{i}Z_{\delta_{m},q_{m}}^{i}.\label{eq:L4}
\end{equation}
 We want to prove that, for all $i=0,\dots,n-1$, $c_{m}^{i}\rightarrow0$
while $m\rightarrow\infty.$ Multiplying equation (\ref{eq:L4}) by
$Z_{\delta_{m},q_{m}}^{j}$ we obtain, by definition of $i^{*}$,
\begin{align*}
\sum_{i=0}^{n-1}c_{m}^{i}\left\langle \left\langle Z_{\delta_{m},q_{m}}^{i},Z_{\delta_{m},q_{m}}^{j}\right\rangle \right\rangle _{H}= & \left\langle \left\langle i^{*}\left(f_{\varepsilon_{m}}'(W_{\delta_{m},q_{m}})[\phi_{\delta_{m},q_{m}}]\right),Z_{\delta_{m},q_{m}}^{j}\right\rangle \right\rangle _{H}\\
= & \int_{\partial M}f_{\varepsilon_{m}}'(W_{\delta_{m},q_{m}})[\phi_{\delta_{m},q_{m}}]Z_{\delta_{m},q_{m}}^{j}d\sigma
\end{align*}
Moreover, by multiplying (\ref{eq:L4}) by $\phi_{\delta_{m},q_{m}}$
we obtain that
\[
\|\phi_{\delta_{m},q_{m}}\|_{H}-\int_{\partial M}f_{\varepsilon_{m}}'(W_{\delta_{m},q_{m}})\phi_{\delta_{m},q_{m}}^{2}d\sigma\rightarrow0,
\]
thus $\left(f_{\varepsilon_{m}}'(W_{\delta_{m},q_{m}})\right)^{1/2}\phi_{\delta_{m},q_{m}}$
is bounded and weakly convergent in $L^{2}(\partial M)$. With this
consideration easily we get 
\begin{multline*}
\int_{\partial M}f_{\varepsilon_{m}}'(W_{\delta_{m},q_{m}})[\phi_{\delta_{m},q_{m}}]Z_{\delta_{m},q_{m}}^{j}d\sigma\\
=\int_{\partial M}\left(f_{\varepsilon_{m}}'(W_{\delta_{m},q_{m}})\right)^{1/2}\phi_{\delta_{m},q_{m}}\left(f_{\varepsilon_{m}}'(W_{\delta_{m},q_{m}})\right)^{1/2}Z_{\delta_{m},q_{m}}^{j}d\sigma\\
=n\int_{\mathbb{R}^{n-1}}U_{1}^{\frac{2}{n-2}}(z,0)\tilde{\phi}(z,0)V_{i}(z,0)dz+o(1)=o(1),
\end{multline*}
once we take in account (\ref{eq:L3}) . 

Now, it is easy to prove that
\[
\left\langle \left\langle Z_{\delta_{m},q_{m}}^{i},Z_{\delta_{m},q_{m}}^{j}\right\rangle \right\rangle _{H}=C\delta_{ij}+o(1),
\]
hence we can conclude that $c_{m}^{i}\rightarrow0$ while $m\rightarrow\infty$
for each $i=0,\dots,n-1$. This, combined with (\ref{eq:L4}) and
since $\|L_{\varepsilon_{m}       d    _{m},q_{m}}(\phi_{\varepsilon_{m}       d    _{m},q_{m}})\|_{\mathcal{H}}\rightarrow0$
gives us that
\begin{equation}
\left\Vert \phi_{\delta_{m},q_{m}}-i^{*}\left(f_{\varepsilon}'(W_{\delta_{m},q_{m}})[\phi_{\delta_{m},q_{m}}]\right)\right\Vert _{\mathcal{H}}=\sum_{i=0}^{n-1}c_{m}^{i}\|Z^{i}\|_{\mathcal{H}}+o(1)=o(1)\label{eq:L5}
\end{equation}
Now, choose a smooth function $\varphi\in C_{0}^{\infty}(\mathbb{R}_{+}^{n})$
and define 
\[
\varphi_{m}(x)=\frac{1}{\delta_{m}^{\frac{n-2}{2}}}\varphi\left(\frac{1}{\delta_{m}}\left(\psi_{q_{m}}^{\partial}\right)^{-1}(x)\right)\chi\left(\left(\psi_{q_{m}}^{\partial}\right)^{-1}(x)\right)\text{ for }x\in M.
\]
We have that $\|\varphi_{m}\|_{H}$ is bounded and, by (\ref{eq:L5}),
that 
\begin{align*}
\left\langle \left\langle \phi_{\delta_{m},q_{m}},\varphi_{m}\right\rangle \right\rangle _{H} 
=& \int_{\partial M}f_{\varepsilon_{m}}'(W_{\delta_{m},q_{m}})
[\phi_{\delta_{m},q_{m}}]\varphi_{m}d\sigma\\
&+\left\langle \left\langle \phi_{\delta_{m},q_{m}}
-i^{*}\left(f_{\varepsilon_{m}}'(W_{\delta_{m},q_{m}})
[\phi_{\delta_{m},q_{m}}]\right),\varphi_{m}\right\rangle \right\rangle _{H}\\
 =& \int_{\partial M}f_{\varepsilon_{m}}'(W_{\delta_{m},q_{m}})
 [\phi_{\delta_{m},q_{m}}]\varphi_{m}d\sigma+o(1)\\
=& (n\pm\varepsilon_{m}(n-2))\int_{\mathbb{R}^{n-1}}
 \frac{1}{\delta_{m}^{\pm\varepsilon\frac{n}{n-2}}}
 U_{1}^{\frac{2}{n-2}\pm\varepsilon_{m}}(z,0)\tilde{\phi}_{m}(z,0)\varphi dz+o(1)\\
=& n\int_{\mathbb{R}^{n-1}}U_{1}^{\frac{2}{n-2}}(z,0)\tilde{\phi}(z,0)\varphi(z,0)dz+o(1),
\end{align*}
by the strong $L_{\text{loc}}^{\frac{2(n-1)}{n-2}}(\partial\mathbb{R}_{+}^{n})$
convergence of $\tilde{\phi}_{m}$. On the other hand 
\[
\left\langle \left\langle \phi_{\delta_{m},q_{m}},\varphi_{m}\right\rangle \right\rangle _{H}=\int_{\mathbb{R}_{+}^{n}}\nabla\tilde{\phi}\nabla\varphi d\eta+o(1),
\]
so $\tilde{\phi}$ is a weak solution of (\ref{eq:PL}) and we conclude
that
\[
\tilde{\phi}\in\text{Span}\left\{ V_{0},V_{1},\dots V_{n}\right\} .
\]
This, combined with (\ref{eq:L3}) gives that $\tilde{\phi}=0$. Proceeding
as before we have
\begin{align*}
\left\langle \left\langle \phi_{\delta_{m},q_{m}},\phi_{\delta_{m},q_{m}}\right\rangle \right\rangle _{H}= & \int_{\partial M}f_{\varepsilon_{m}}'(W_{\delta_{m},q_{m}})[\phi_{\delta_{m},q_{m}}]\phi_{\delta_{m},q_{m}}d\sigma+o(1)\\
=(n\pm\varepsilon_{m}(n-2)) & \int_{\mathbb{R}^{n-1}}\frac{1}{\delta_{m}^{\pm\varepsilon\frac{n}{n-2}}}U_{1}^{\frac{2}{n-2}\pm\varepsilon_{m}}(z,0)\tilde{\phi}_{m}^{2}(z,0)\varphi dz+o(1)=o(1)
\end{align*}
In a similar way, by (\ref{eq:L5}) we have 
\[
\left|\phi_{\delta_{m},q_{m}}\right|_{L^{s_{\varepsilon}}}=\left|i^{*}\left(f_{\varepsilon}'(W_{\delta_{m},q_{m}})[\phi_{\delta_{m},q_{m}}]\right)\right|_{L^{s_{\varepsilon}}}+o(1)=o(1)
\]
which gives $\left\Vert \phi_{\delta_{m},q_{m}}\right\Vert _{\mathcal{H}}\rightarrow0$
that is a contradiction.\end{proof}

\subsection{The estimate of the error term $R$}

\begin{lem}
\label{lem:R}  For $a,b\in\mathbb{R}$,
$0<a<b$ there exists a positive constant $C_{1}=C_{1}(a,b)$ such
that, for $\varepsilon$ small, for any $q\in\partial M$ and for
any $       d    \in[a,b]$ there holds
\[
\|R_{\varepsilon,\delta,q}\|_{\mathcal{H}}\le C_{1}\varepsilon\left|\ln\varepsilon\right|
\]
\end{lem}
\begin{proof}
We estimate
\begin{align*}
\left\Vert i^{*}\left(f_{\varepsilon}(W_{\delta,q}(x)\right)-W_{\delta,q}(x)\right\Vert _{H} & \le\left\Vert i^{*}\left(f_{\varepsilon}(W_{\delta,q}(x)\right)-i^{*}\left(f_{0}(W_{\delta,q}(x)\right)\right\Vert _{H}\\
 & +\left\Vert i^{*}\left(f_{0}(W_{\delta,q}(x)\right)-W_{\delta,q}(x)\right\Vert _{H}.
\end{align*}
By definiton of $i^{*}$ there exists $\Gamma$ which solves the equation
\begin{equation}
\left\{ \begin{array}{ll}
-\Delta_{g}\Gamma+a(x)\Gamma=0 & \text{ on }M\\
\frac{\partial}{\partial\nu}\Gamma+\frac{n-2}{2}b(x)\Gamma=f_{0}(W_{\delta,q}) & \text{ on }\partial M
\end{array}\right..\label{eq:gamma}
\end{equation}
so, by (\ref{eq:gamma}), we have
\begin{align*}
\Vert i^{*}\left(f_{0}(W_{\delta,q}(x)\right)&-W_{\delta,q}(x)\Vert _{H}=  \|\Gamma(x)-W_{\delta,q}(x)\|_{H}^{2}\\
= & \int_{M}\left[-\Delta_{g}(\Gamma-W_{\delta,q})+a(\Gamma-W_{\delta,q})\right](\Gamma-W_{\delta,q})d\mu_{g}\\
 & +\int_{\partial M}\left[\frac{\partial}{\partial\nu}(\Gamma-W_{\delta,q})+\frac{(n-2)}{2}b(x)(\Gamma-W_{\delta,q})\right](\Gamma-W_{\delta,q})d\mu_{g}\\
= & \int_{M}\left[\Delta_{g}W_{\delta,q}-aW_{\delta,q}\right](\Gamma-W_{\delta,q})d\mu_{g}\\
 & +\int_{\partial M}\left[f_{0}(W_{\delta,q})-\frac{\partial}{\partial\nu}W_{\delta,q}\right](\Gamma-W_{\delta,q})d\mu_{g}\\
 & -\frac{(n-2)}{2}\int_{\partial M}b(x)W_{\delta,q}(\Gamma-W_{\delta,q})d\mu_{g}:=I_{1}+I_{2}+I_{3}
\end{align*}
We obtain
\begin{equation}
I_{1}=\left\Vert \Gamma-W_{\delta,q}\right\Vert _{H}O(\delta).\label{eq:I1}
\end{equation}
Infact

\[
I_{1}\le\left|\Delta_{g}W_{\delta,q}-aW_{\delta,q}\right|_{L^{\frac{2n}{n+2}}(M)}\left|\Gamma-W_{\delta,q}\right|_{L^{\frac{2n}{n-2}}(M)}\le\left|\Delta_{g}W_{\delta,q}-aW_{\delta,q}\right|_{L^{\frac{2n}{n+2}}(M)}\left\Vert \Gamma-W_{\delta,q}\right\Vert _{H}.
\]
Easily we have that $\left|W_{\delta,q}\right|_{L^{\frac{2n}{n+2}}}=O(\delta^{2})$.
For the other term we have, in coordinates, 
\begin{equation}
\Delta_{g}W_{\delta,q}=\Delta[U_{\delta}\chi]+(g^{ab}-\delta_{ab})\partial_{ab}[U_{\delta}\chi]-g^{ab}\Gamma_{ab}^{k}\partial_{k}[U_{\delta}\chi],\label{eq:restoarm1}
\end{equation}
$\Gamma_{ab}^{k}$ being the Christoffel symbols. Using the expansion
of the metric $g^{ab}$ given by (\ref{eq:g1}) and (\ref{eq:g2})
we have that 
\begin{equation}
\left|(g^{ab}-\delta_{ab})\partial_{ab}[U_{\delta}\chi]\right|_{L^{\frac{2n}{n+2}}(M)}=O(\delta)\text{ and }\left|g^{ab}\Gamma_{ab}^{k}\partial_{k}[U_{\delta}\chi]\right|_{L^{\frac{2n}{n+2}}(M)}=O(\delta^{2})\label{eq:restoarm2}
\end{equation}
 Since $U_{\delta}$ is a harmonic function we deduce
\begin{equation}
\left|\Delta[U_{\delta}\chi]\right|_{L^{\frac{2n}{n+2}}(M)}=\left|U_{\delta}\Delta\chi+2\nabla U_{\delta}\nabla\chi]\right|_{L^{\frac{2n}{n+2}}(M)}=O(\delta^{2}).\label{eq:restoarm3}
\end{equation}
For the second integral $I_{2}$ we have
\begin{equation}
I_{2}=\left\Vert \Gamma-W_{\delta,q}\right\Vert _{H}O(\delta^{2}).\label{eq:I2}
\end{equation}
since
\begin{align*}
I_{2} & \le\left|f_{0}(W_{\delta,q})-\frac{\partial}{\partial\nu}W_{\delta,q}\right|_{L^{\frac{2(n-1)}{n}}(\partial M)}\left|\Gamma-W_{\delta,q}\right|_{L^{\frac{2(n-1)}{n-2}}(\partial M)}\\
 & \le C\left|f_{0}(W_{\delta,q})-\frac{\partial}{\partial\nu}W_{\delta,q}\right|_{L^{\frac{2(n-1)}{n}}(\partial M)}\left\Vert \Gamma-W_{\delta,q}\right\Vert _{H},
\end{align*}
and, using the boundary condition for (\ref{eq:PL}) we have 
\begin{multline}
\left|f_{0}(W_{\delta,q})-\frac{\partial}{\partial\nu}W_{\delta,q}\right|_{L^{\frac{2(n-1)}{n}}(\partial M)}\\
=\frac{1}{\delta^{\frac{n}{2}}}\left(\int_{\mathbb{R}^{n-1}}|g(\delta z,0)|^{\frac{1}{2}}\left[(n-2)U^{\frac{n}{n-2}}(z,0)\chi^{\frac{n}{n-2}}(\delta z,0)-\chi(\delta z,0)\frac{\partial U}{\partial t}(z,0)\right]^{\frac{2(n-1)}{n}}\delta^{n-1}dz\right)^{\frac{n}{2(n-1)}}\\
\le C\left(\int_{\mathbb{R}^{n-1}}\left[(n-2)U^{\frac{n}{n-2}}(z,0)\left[\chi^{\frac{n}{n-2}}(\delta z,0)-\chi(\delta z,0)\right]\right]^{\frac{2(n-1)}{n}}dz\right)^{\frac{n}{2(n-1)}}=O(\delta^{2}),\label{eq:restobordo}
\end{multline}
Lastly, 
\begin{equation}
I_{3}\le\left|W_{\delta,q}\right|_{L^{\frac{2(n-1)}{n}}(\partial M)}\left|\Gamma-W_{\delta,q}\right|_{L^{\frac{2(n-1)}{n-2}}(\partial M)}=\left\Vert \Gamma-W_{\delta,q}\right\Vert _{H}O(\delta).\label{eq:I3}
\end{equation}
By (\ref{eq:I1}), (\ref{eq:I2}) and (\ref{eq:I3}) we conclude that
\[
\left\Vert i^{*}\left(f_{0}(W_{\delta,q}(x)\right)-W_{\delta,q}(x)\right\Vert _{H}=\|\Gamma(x)-W_{\delta,q}(x)\|_{H}=O(\delta).
\]
To conclude the proof we estimate the term $\left\Vert i^{*}\left(f_{\varepsilon}(W_{\delta,q}(x)\right)-i^{*}\left(f_{0}(W_{\delta,q}(x)\right)\right\Vert _{H}$.
We have, by the properties of $i^{*}$, that 
\begin{multline*}
\left\Vert i^{*}\left(f_{\varepsilon}(W_{\delta,q}(x)\right)-i^{*}\left(f_{0}(W_{\delta,q}(x)\right)\right\Vert _{H}\le\left|W_{\delta,q}(x)^{\frac{n}{n-2}\pm\varepsilon}-W_{\delta,q}^{\frac{n}{n-2}}(x)\right|_{L^{\frac{2(n-1)}{n}}(\partial M)}\\
=\left\{ \int_{\mathbb{R}^{n-1}}\left[\left(\frac{1}{\delta^{\pm\varepsilon\frac{n-2}{2}}}U^{\pm\varepsilon}(z,0)-1\right)U^{\frac{n}{n-2}}(z,0)\right]^{\frac{2(n-1)}{n}}dz\right\} ^{\frac{n}{2(n-1)}}+O(\delta^{2})
\end{multline*}
To estimate the last integral, we first recall two Taylor espansions
with respect to $\varepsilon$
\begin{align}
U^{\pm\varepsilon} & =1\pm\varepsilon\ln U+\frac{1}{2}\varepsilon^{2}\ln^{2}U+o(\varepsilon^{2})\label{eq:Uallaeps}\\
\delta^{\mp\varepsilon\frac{n-2}{2}} & =1\mp\varepsilon\frac{n-2}{2}\ln\delta+\varepsilon^{2}\frac{(n-2)^{2}}{8}\ln^{2}\delta+o(\varepsilon^{2}\ln^{2}\delta)\label{eq:deltaallaeps}
\end{align}
In light of (\ref{eq:Uallaeps}) and (\ref{eq:deltaallaeps}) we have
\begin{multline}
\left\Vert i^{*}\left(f_{\varepsilon}(W_{\delta,q})\right)-i^{*}\left(f_{0}(W_{\delta,q})\right)\right\Vert _{H}\\
\le\left\{ \int_{\mathbb{R}^{n-1}}\left|\left(\mp\frac{n-2}{2}\varepsilon\ln\delta\pm\varepsilon\ln U(z,0)+O(\varepsilon^{2})+O(\varepsilon^{2}\ln\delta)\right)U^{\frac{n}{n-2}}(z,0)\right|^{\frac{2(n-1)}{n}}dz\right\} ^{\frac{n}{2(n-1)}}+O(\delta^{2})\\
=\frac{n-2}{2}\varepsilon\ln\delta\left|U(z,0)\right|_{L^{\frac{2(n-1)}{n-2}}(\mathbb{R}^{n-1})}^{\frac{n}{n-2}}+\varepsilon\left\{ \int_{\mathbb{R}^{n-1}}U^{\frac{2(n-1)}{n-2}}(z,0)\ln U(z,0)dz\right\} ^{\frac{n}{2(n-1)}}\\
+O(\varepsilon^{2})+O(\varepsilon^{2}\left|\ln\delta\right|)+O(\delta^{2})\\
=O(\varepsilon)+O(\varepsilon\left|\ln\delta\right|)+O(\delta^{2}).\label{eq:R3}
\end{multline}
Choosing $\delta=       d    \varepsilon$ concludes the proof of Lemma
\ref{lem:R} for the subcritical case.

For the supercritical case, we have to control $|R_{\varepsilon,\delta,q}|_{L^{s_{\varepsilon}}(\partial M)}$.
As in the previous case we consider 
\begin{align*}
|R_{\varepsilon,\delta,q}|_{L^{s_{\varepsilon}}(\partial M)}\le & \left|i^{*}\left(f_{\varepsilon} (W_{\delta,q}(x)\right)-i^{*}\left(f_{0}(W_{\delta,q}(x)\right)\right|_{L^{s_{\varepsilon}}(\partial M)}\\
 & +\left|i^{*}\left(f_{0}(W_{\delta,q}(x)\right)-W_{\delta,q}(x)\right|_{L^{s_{\varepsilon}}(\partial M)}.
\end{align*}
As before, set $\Gamma=i^{*}\left(f_{0}(W_{\delta,q}(x)\right)$.
Since $\Gamma$ solves (\ref{eq:gamma}), $\Gamma-W_{\delta,q}$ solves
\[
\left\{ \begin{array}{ll}
-\Delta_{g}(\Gamma-W_{\delta,q})+a(x)(\Gamma-W_{\delta,q})=-\Delta_{g}W_{\delta,q}+a(x)W_{\delta,q} & \text{ on }M\\
\frac{\partial}{\partial\nu}(\Gamma-W_{\delta,q})+\frac{n-2}{2}b(x)(\Gamma-W_{\delta,q})=f_{0}(\Gamma)+\frac{\partial}{\partial\nu}W_{\delta,q}+\frac{n-2}{2}b(x)W_{\delta,q} & \text{ on }\partial M
\end{array}\right..
\]
We choose $q$ as in (\ref{eq:nittka1}), and $r=\varepsilon$, thus,
by Theorem 3.14 in \cite{N},
 we have 
\begin{align*}
|\Gamma-W_{\delta,q}|_{L^{s_{\varepsilon}}(\partial M)}\le & |-\Delta_{g}W_{\delta,q}+a(x)W_{\delta,q}|_{L^{q+\varepsilon}(M)}\\
 & +\left|f_{0}(\Gamma)+\frac{\partial}{\partial\nu}W_{\delta,q}+\frac{n-2}{2}b(x)W_{\delta,q}\right|_{L^{\frac{(n-1)q}{n-q}+\varepsilon}(\partial M)}.
\end{align*}
We remark that 
\[
q=\frac{2n+n^{2}\left(\frac{n-2}{n-1}\right)\varepsilon}{n+2+2n\left(\frac{n-2}{n-1}\right)\varepsilon}=\frac{2n}{n+2}+O^{+}(\varepsilon)\text{ with }0<O^{+}(\varepsilon)<C\varepsilon
\]
for some positive constant $C$. By direct computation we have
\begin{align*}
|a(x)W_{\delta,q}|_{L^{q+\varepsilon}(M)} & \le C\delta^{2-O^{+}(\varepsilon)};\\
\left|b(x)W_{\delta,q}\right|_{L^{\frac{(n-1)q}{n-q}+\varepsilon}(\partial M)} & \le C\delta^{1-O^{+}(\varepsilon)}.
\end{align*}
Moreover, proceeding as in (\ref{eq:restoarm1}),(\ref{eq:restoarm2}),
(\ref{eq:restoarm3}), and as in (\ref{eq:restobordo}) we get
\begin{align*}
|\Delta_{g}W_{\delta,q}|_{L^{q+\varepsilon}(M)} & \le C\delta^{2-O^{+}(\varepsilon)};\\
\left|f_{0}(\Gamma)+\frac{\partial}{\partial\nu}W_{\delta,q}\right|_{L^{\frac{(n-1)q}{n-q}+\varepsilon}(\partial M)} & \le C\delta^{1-O^{+}(\varepsilon)}.
\end{align*}
Since $i^{*}\left(f_{\varepsilon} (W_{\delta,q})\right)$ solves
(\ref{eq:1.2}), and $i^{*}\left(f_{\varepsilon} |u|^{{n\over n-2}+\varepsilon}(W_{\delta,q})\right)$
solves (\ref{eq:1.2}), we again use Theorem 3.14 in \cite{N}. Taking
in account (\ref{eq:Uallaeps}) and (\ref{eq:deltaallaeps}) finally
we get 
\begin{multline}
\left|i^{*}\left(f_{\varepsilon}(W_{\delta,q})\right)-i^{*}\left(f_{0}(W_{\delta,q})\right)\right|_{L^{s_{\varepsilon}}(\partial M)}\le\left|f_{\varepsilon}(W_{\delta,q})-f_{0}(W_{\delta,q})\right|_{L^{\frac{2(n-1)}{n}+O^{+}(\varepsilon)}(\partial M)}\\
\le\delta^{-O^{+}(\varepsilon)}\left\{ \int_{\mathbb{R}^{n-1}}\left[\left(\frac{1}{\delta^{\varepsilon\frac{n-2}{2}}}U^{\varepsilon}(z,0)-1\right)U^{\frac{n}{n-2}}(z,0)\right]^{\frac{2(n-1)}{n}+O^{+}(\varepsilon)}dz\right\} ^{\frac{1}{\frac{2(n-1)}{n}+O^{+}(\varepsilon)}}+O(\delta^{2})\\
=\delta^{-O^{+}(\varepsilon)}\left\{ O(\varepsilon\left|\ln\delta\right|)+O(\varepsilon)\right\} +O(\delta^{2}).\label{eq:R3sup}
\end{multline}
Now, choosing $\delta=       d    \varepsilon$, we can conlcude the proof,
since 
\[
\delta^{-O^{+}(\varepsilon)}=1+O^{+}(\varepsilon)\left|\ln(\varepsilon       d    )\right|=1+O^{+}(\varepsilon\left|\ln\varepsilon\right|)=O(1).
\]
\end{proof}

\subsection{Solving equation \eqref{eq:rid1}: the remainder term $\phi$}
\begin{prop}
\label{prop:phi}  For $a,b\in\mathbb{R}$,
$0<a<b$ there exists a positive constant $C=C(a,b)$ such that, for
$\varepsilon$ small, for any $q\in\partial M$ and for any $       d    \in[a,b]$
there exists a unique $\phi_{\delta,q}$ which solves (\ref{eq:rid1})
\[
\|\phi_{\delta,q}\|_{\mathcal{H}}\le C\varepsilon\left|\ln\varepsilon\right|.
\]
Moreover the map $q\mapsto\phi_{\delta,q}$ is a $C^{1}(\partial M,\mathcal{H})$
map.\end{prop}
\begin{proof}
First of all, we point out that   $N$ is a contraction mapping.
 We remark that the conjugate exponent
of $s_{\varepsilon}$ is 
\[
s_{\varepsilon}^{'}=\left\{ \begin{array}{cc}
\frac{2(n-1)}{n} & \text{in the subcritical case}\\
\\
\frac{2(n-1)+\varepsilon n(n-2)}{n+\varepsilon n(n-2)} & \text{in the supercritical case}
\end{array}\right..
\]
By the properties of $i^{*}$ and using the expansion of $f_{\varepsilon}(W_{\delta,q}+\phi_{1})$
centered in$W_{\delta,q}+\phi_{2}$ we have 
\begin{align*}
\|N(\phi_{1})-N(\phi_{2})\|_{\mathcal{H}}\le & \|f_{\varepsilon}(W_{\delta,q}+\phi_{1})-f_{\varepsilon}(W_{\delta,q}+\phi_{2})-f_{\varepsilon}'(W_{\delta,q})[\phi_{1}-\phi_{2}]\|_{L^{s_{\varepsilon}'}(\partial M)}\\
\le & \left\Vert \left(f_{\varepsilon}'\left(W_{\delta,q}+\theta\phi_{1}+(1-\theta)\phi_{2}\right)-f_{\varepsilon}'(W_{\delta,q})\right)[\phi_{1}-\phi_{2}]\right\Vert _{L^{s_{\varepsilon}'}(\partial M)}
\end{align*}
and, since $|\phi_{1}-\phi_{2}|^{s_{\varepsilon}'}\in L^{s_{\varepsilon}/s_{\varepsilon}'}(\partial M)$
and $|f_{\varepsilon}'(\cdot)|^{s_{\varepsilon}'}\in L^{\left(\frac{s_{\varepsilon}}{s_{\varepsilon}'}\right)^{\prime}}(\partial M)$
since $f_{\varepsilon}'(\cdot)\in L^{s_{\varepsilon}}(\partial M)$,
we have 
\begin{multline*}
\|N(\phi_{1})-N(\phi_{2})\|_{\mathcal{H}}\\
\le\left\Vert \left(f_{\varepsilon}'\left(W_{\delta,q}+\theta\phi_{1}+(1-\theta)\phi_{2}\right)-f_{\varepsilon}'(W_{\delta,q})\right)\right\Vert _{L^{s_{\varepsilon}}(\partial M)}\|\phi_{1}-\phi_{2}\|_{L^{s_{\varepsilon}}(\partial M)}\\
=\gamma\|\phi_{1}-\phi_{2}\|_{\mathcal{H}}
\end{multline*}
where 
\[
\gamma=\left\Vert \left(f_{\varepsilon}'\left(W_{\delta,q}+\theta\phi_{1}+(1-\theta)\phi_{2}\right)-f_{\varepsilon}'(W_{\delta,q})\right)\right\Vert _{L^{s_{\varepsilon}}(\partial M)}<1
\]
provided $\|\phi_{1}\|_{\mathcal{H}}$ and $\|\phi_{2}\|_{\mathcal{H}}$
sufficiently small. 

In the same way we can prove that 
$
\|N(\phi)\|_{\mathcal{H}}\le\gamma\|\phi\|_{\mathcal{H}}
$
with $\gamma<1$ if $\|\phi\|_{\mathcal{H}}$ is sufficiently small.\\

 Next, 
by Lemma \ref{lem:L} and by Lemma \ref{lem:R} we have
\[
\|L^{-1}(N(\phi)+R_{\varepsilon,\delta,q})\|_{\mathcal{H}}\le C\left(\gamma\|\phi\|_{\mathcal{H}}+\varepsilon\left|\ln\varepsilon\right|\right)
\]
 where $C=\max\{C_{0},C_{0}C_{1}\}>0$ being $C_{0},C_{1}$ the constants
which appear in Lemma \ref{lem:L} and in Lemma \ref{lem:R}. Notice
that, given $C>0$,  it is possible (up to choose
$\|\phi\|_{\mathcal{H}}$ sufficiently small) to choose $0<C\gamma<1/2$. 

Now, if $\|\phi\|_{\mathcal{H}}\le2C\varepsilon\left|\ln\varepsilon\right|$,
then the map
\[
T(\phi):=L^{-1}(N(\phi)+R_{\varepsilon,\delta,q})
\]
is a contraction from the ball $\|\phi\|_{\mathcal{H}}\le2C\varepsilon\left|\ln\varepsilon\right|$
in itself, so, by the fixed point Theorem, there exists a unique $\phi_{\delta,q}$
with $\|\phi_{\delta,q}\|_{\mathcal{H}}\le2C\varepsilon\left|\ln\varepsilon\right|$
solving (\ref{eq:rid1equiv}) and hence (\ref{eq:rid1}). The regularity
of the map $q\mapsto\phi_{\delta,q}$ can be proven via the implicit
function Theorem.
\end{proof}

\section{The reduced problem}\label{tre}

Problem (\ref{eq:1.2}) has a variational structure. Weak solutions
to (\ref{eq:1.2}) are critical points of the energy functional $J_\varepsilon:{\mathcal H}\to\mathbb R$ 
\begin{multline*}
J_{\varepsilon}(u)=\frac{1}{2}\int_{M}\left(|\nabla u|^{2}+a(x)u^{2}\right)d\mu_{g}+\frac{n-2}{4}\int_{\partial M}b(x)u^{2}d\sigma\\
-\frac{(n-2)^{2}}{2n-2\pm\varepsilon(n-2)}\int_{\partial M}u^{\frac{2n-2}{n-2}\pm\varepsilon}d\sigma
\end{multline*}

 Let us introduce the reduced energy $I_\varepsilon:(0,+\infty)\times\partial M\to\mathbb R$ by 
\begin{equation}\label{rido}
I_{\varepsilon}(       d    ,q):=J_{\varepsilon}(W_{\varepsilon       d    ,q}+\phi_{\varepsilon       d    ,q}) .
\end{equation}
where the remainder term $\phi_{\varepsilon       d    ,q}$ has been found in \eqref{prop:phi}.

 \subsection{The reduced energy}
 
Here we will use the following expansion for the metric tensor on
$M$. 
\begin{eqnarray}
g^{ij}(y) & = & \delta_{ij}+2h_{ij}(0)y_{n}+O(|y|^{2})\text{ for }i,j=1,\dots n-1\label{eq:g1}\\
g^{in}(y) & = & \delta_{in}\text{ for }i=1,\dots n-1\label{eq:g2}\\
\sqrt{g}(y) & = & 1-(n-1)H(0)y_{n}+O(|y|^{2})\label{eq:g3}
\end{eqnarray}
where $(y_{1},\dots,y_{n})$ are the Fermi coordinates and, by definition
of $h_{ij}$, 
\begin{equation}
H=\frac{1}{n-1}\sum_{i}^{n-1}h_{ii}.\label{eq:H}
\end{equation}
We also recall that on $\partial M$ the Fermi coordinates coincide
with the exponential ones, so we have that
\begin{equation}
\sqrt{g}(y_{1},\dots,y_{n-1},0)=1+O(|y|^{2}).\label{eq:gexp}
\end{equation}
To improve the readability of this paper, thereafter we will introduce
$z=(z_{1},\dots,z_{n-1})$ to indicate the first $n-1$ Fermi coordinates
and $t$ to indicate the last one, so $(y_{1},\dots,y_{n-1},y_{n})=(z,t)$.
Moreover, indices $i,j$ conventionally refer to sums from $1$
to $n-1$, while $l,m$ usually refer to sums from $1$ to $n$. 

\begin{prop}\label{expan}

\begin{itemize}
\item[(i)] If $(       d    _{0},q_{0})\in(0,+\infty)\times\partial M$
is a critical point for the reduced energy $I_\varepsilon$ defined in \eqref{rido},
then $W_{\varepsilon       d _0   ,q_0}+\phi_{\varepsilon       d_0    ,q_0}\in\mathcal{H}$
solves problem (\ref{eq:1.2}).

\item[(ii)] It holds true that
$$ 
I_\varepsilon(d,q)=c_n(\varepsilon)+\varepsilon\left[\alpha_n d\varphi(q)-\beta_n\ln d\right]+o(\varepsilon)\ \hbox{in the subcritical case}
$$ 
and
$$ 
I_\varepsilon(d,q)=c_n(\varepsilon)+\varepsilon\left[\alpha_n d\varphi(q)+\beta_n\ln d\right]+o(\varepsilon)\ \hbox{in the supercritical case}
$$ 
$C^0-$uniformly with respect to $d$ in compact sets of $(0,+\infty)$ and $q\in\partial M.$
Here $c_n(\varepsilon)$ is a constant which only depends on $\varepsilon $ and $n$,  $\alpha_n$ and $\beta_n$ are positive constants which only depend  on $n $
and  
 $\varphi(q)=h(q)-H_g(q)$  is the function defined in \eqref{diff} .
\end{itemize}
\end{prop}

\begin{proof}

{\underline{\em Proof of (i)}}.

Set $q:=q(y)=\psi_{q_{0}}^{\partial}(y).$ Since $(       d    _{0},q_{0})$
is a critical point, we have, for any $h\in1,\dots n-1$,
\begin{align*}
0 & =\left.\frac{\partial}{\partial y_{h}}I_{\varepsilon}(       d    ,\psi_{q_{0}}^{\partial}(y))\right|_{y=0}\\
 & =\left.\left\langle \left\langle W_{\varepsilon       d    ,q(y)}+\phi_{\varepsilon       d    ,q(y)}-i^{*}(f_{\varepsilon}(W_{\varepsilon       d    ,q(y)}+\phi_{\varepsilon       d    ,q(y)})),\frac{\partial}{\partial y_{h}}W_{\varepsilon       d    ,q(y)}+\frac{\partial}{\partial y_{h}}\phi_{\varepsilon       d    ,q(y)}\right\rangle \right\rangle _{H}\right|_{y=0}\\
 & =\sum_{i=0}^{n-1}c_{\varepsilon}^{i}\left.\left\langle \left\langle Z_{\varepsilon       d    ,q(y)}^{i},\frac{\partial}{\partial y_{h}}W_{\varepsilon       d    ,q(y)}+\frac{\partial}{\partial y_{h}}\phi_{\varepsilon       d    ,q(y)}\right\rangle \right\rangle _{H}\right|_{y=0}\\
 & =\sum_{i=0}^{n-1}c_{\varepsilon}^{i}\left.\left\langle \left\langle Z_{\varepsilon       d    ,q(y)}^{i},\frac{\partial}{\partial y_{h}}W_{\varepsilon       d    ,q(y)}\right\rangle \right\rangle _{H}\right|_{y=0}-\sum_{i=0}^{n-1}c_{l}^{i}\left.\left\langle \left\langle \frac{\partial}{\partial y_{h}}Z_{\varepsilon       d    ,q(y)}^{i},\phi_{\varepsilon       d    ,q(y)}\right\rangle \right\rangle _{H}\right|_{y=0}
\end{align*}
using that $\phi_{\varepsilon       d    ,q(y)}$ is a solution of (\ref{eq:rid1})
and that $\left\langle \left\langle Z_{\varepsilon       d    ,q(y)}^{i},\frac{\partial}{\partial y_{h}}\phi_{\varepsilon       d    ,q(y)}\right\rangle \right\rangle =-\left\langle \left\langle \frac{\partial}{\partial y_{h}}Z_{\varepsilon       d    ,q(y)}^{i},\phi_{\varepsilon       d    ,q(y)}\right\rangle \right\rangle $
since $\phi_{\varepsilon       d    ,q(y)}\in K_{\varepsilon       d    ,q(y)}^{\bot}$
for all $y$. Now it is enough to observe that 
\[
\left\langle \left\langle \frac{\partial}{\partial y_{h}}Z_{\varepsilon       d    ,q(y)}^{i},\phi_{\varepsilon       d    ,q(y)}\right\rangle \right\rangle _{H}\le\left\Vert \frac{\partial}{\partial y_{h}}Z_{\varepsilon       d    ,q(y)}^{i}\right\Vert _{H}\left\Vert \phi_{\varepsilon       d    ,q(y)}\right\Vert _{H}=o(1)
\]
 
\[
\left\langle \left\langle Z_{\varepsilon       d    ,q(y)}^{i},\frac{\partial}{\partial y_{h}}W_{\varepsilon       d    ,q(y)}\right\rangle \right\rangle _{H}=\frac{1}{\varepsilon       d    }\left\langle \left\langle Z_{\varepsilon       d    ,q(y)}^{i},Z_{\varepsilon       d    ,q(y)}^{h}\right\rangle \right\rangle _{H}=\frac{1}{\varepsilon       d    }\delta^{ih}+o(1)
\]
to conclude that
\[
0=\frac{1}{\varepsilon       d    }\sum_{i=0}^{n-1}c_{\varepsilon}^{i}(\delta^{ih}+o(1))
\]
and so $ $$c_{\varepsilon}^{i}=0$ for all $i=0,\dots,n-1$. This
conclude the proof.
 \\

{\underline{\em Proof of (ii)}}.

{\em Step 1: we prove that 
for $\varepsilon$ small enough and for any $q\in\partial M$,
\[
\left|J_{\varepsilon}(W_{\delta,q}+\phi_{\delta,q})-J_{\varepsilon}(W_{\delta,q})\right|\le\|\phi_{\delta,q}\|_{\mathcal{H}}^{2}+C\varepsilon\left|\ln\varepsilon\right|\|\phi_{\delta,q}\|_{\mathcal{H}}=o(\varepsilon)
\]
}\\

We have 
\begin{multline*}
\left|J_{\varepsilon}(W_{\delta,q}+\phi_{\delta,q})-J_{\varepsilon}(W_{\delta,q})\right|=\left|\int_{M}\left[-\Delta_{g}W_{\delta,q}+a(x)W_{\delta,q}\right]\phi_{\delta,q}d\mu_{g}\right|+\frac{1}{2}\|\phi_{\delta,q}\|_{H}^{2}\\
+\left|\int_{\partial M}\left[\frac{\partial}{\partial\nu}W_{\delta,q}+\frac{n-2}{2}b(x)W_{\delta,q}-f_{0}(W_{\delta,q})\right]\phi_{\delta,q}d\sigma\right|+\left|\int_{\partial M}\left[f_{0}(W_{\delta,q})-f_{\varepsilon}(W_{\delta,q})\right]\phi_{\delta,q}d\sigma\right|\\
+\left|\int_{\partial M}\frac{(n-2)^{2}}{2n-2\pm\varepsilon(n-2)}\left[(W_{\delta,q}+\phi_{\delta,q})^{\frac{2n-2}{n-2}\pm\varepsilon}-W_{\delta,q}^{\frac{2n-2}{n-2}\pm\varepsilon}\right]-f_{\varepsilon}(W_{\delta,q})\phi_{\delta,q}d\sigma\right|.
\end{multline*}
With the same estimate of $I_{1}$ in Lemma \ref{lem:R} we obtain
that 
\[
\left|\int_{M}\left[-\Delta_{g}W_{\delta,q}+a(x)W_{\delta,q}\right]\phi_{\delta,q}d\mu_{g}\right|=O(\delta)\|\phi_{\delta,q}\|_{H},
\]
and in light of the estimate of $I_{2}$ and $I_{3}$ in Lemma \ref{lem:R}
we get 
\[
\left|\int_{\partial M}\left[\frac{\partial}{\partial\nu}W_{\delta,q}+\frac{n-2}{2}b(x)W_{\delta,q}-f_{0}(W_{\delta,q})\right]\phi_{\delta,q}d\sigma\right|=O(\delta)\|\phi_{\delta,q}\|_{H}.
\]
In the subcritical case, following the computation in (\ref{eq:R3})
we obtain
\begin{multline*}
\left|\int_{\partial M}\left[f_{0}(W_{\delta,q})-f_{\varepsilon}(W_{\delta,q})\right]\phi_{\delta,q}d\sigma\right|\\
\le C\left|f_{0}(W_{\delta,q})-f_{\varepsilon}(W_{\delta,q})\right|_{L^{\frac{2(n-1)}{n}}(\partial M)}\left|\phi_{\delta,q}\right|_{L^{\frac{2(n-1)}{n-2}}(\partial M)}\\
=\left[O(\varepsilon)+O(\varepsilon\ln\delta)\right]\|\phi_{\delta,q}\|_{H}=O(\varepsilon\left|\ln\varepsilon\right|)\|\phi_{\delta,q}\|_{H}
\end{multline*}
and in a similar way, for the supercritical case we get, in light
of (\ref{eq:R3sup})
\begin{multline*}
\left|\int_{\partial M}\left[f_{0}(W_{\delta,q})-f_{\varepsilon}(W_{\delta,q})\right]\phi_{\delta,q}d\sigma\right|\\
\le C\left|f_{0}(W_{\delta,q})-f_{\varepsilon}(W_{\delta,q})\right|_{L^{\frac{2(n-1)}{n}+O^{+}(\varepsilon)}(\partial M)}\left|\phi_{\delta,q}\right|_{L^{\frac{2(n-1)}{n-2}-O^{+}(\varepsilon)}(\partial M)}\\
\le\left(\delta^{-O^{+}(\varepsilon)}\left\{ O(\varepsilon\ln\delta)+O(\varepsilon)\right\} +O(\delta^{2})\right)\|\phi_{\delta,q}\|_{H}=O(\varepsilon\left|\ln\varepsilon\right|)\|\phi_{\delta,q}\|_{H}
\end{multline*}
Finally, using taylor expansion formula we have immediately, for some
$\theta\in(0,1)$,
\begin{multline*}
\left|\int_{\partial M}\frac{(n-2)^{2}}{2n-2\pm\varepsilon(n-2)}\left[(W_{\delta,q}+\phi_{\delta,q})^{\frac{2n-2}{n-2}\pm\varepsilon}-W_{\delta,q}^{\frac{2n-2}{n-2}\pm\varepsilon}\right]-f_{\varepsilon}(W_{\delta,q})\phi_{\delta,q}d\sigma\right|\\
=\left|\frac{n\pm\varepsilon(n-2)}{2}\int_{\partial M}(W_{\delta,q}+\theta\phi_{\delta,q})^{\frac{2}{n-2}\pm\varepsilon}\phi_{\delta,q}^{2}d\sigma\right|\\
\le C\left[\int_{\partial M}|W_{\delta,q}+\theta\phi_{\delta,q}|^{\left(\frac{2}{n-2}\pm\varepsilon\right)\frac{s_{\varepsilon}}{s_{\varepsilon}-2}}d\sigma\right]^{\frac{s_{\varepsilon}-2}{s_{\varepsilon}}}\left[\int_{\partial M}|\phi_{\delta,q}|^{s_{\varepsilon}}d\sigma\right]^{\frac{2}{s_{\varepsilon}}}\\
\le C|W_{\delta,q}+\theta\phi_{\delta,q}|_{L^{s_{\varepsilon}}}^{s_{\varepsilon}-2}\|\phi_{\delta,q}\|_{\mathcal{H}}^{2}\le C\|\phi_{\delta,q}\|_{\mathcal{H}}^{2}.
\end{multline*}
Choosing $\delta=       d    \varepsilon$, and recalling that, by Proposition
\ref{prop:phi}, $\|\phi_{\delta,q}\|_{\mathcal{H}}=O(\varepsilon|\ln\varepsilon|)$
concludes the proof.
 \bigskip

{\em  Step 2: we prove that 
\begin{align*}
J_{\varepsilon}(W_{\delta,q}) & =C(\varepsilon)+\varepsilon\left\{        d    \frac{n-2}{4}\left[b(q)-H(q)\right]\pm\ln       d    \frac{(n-2)^{3}(n-3)}{4(n-2)(2n-2)}\right\} \omega_{n-1}I_{n-2}^{n-2}+o(\varepsilon)
\end{align*}
$C^{0}$-uniformly with respect to $d$  in compact sets of $(0,+\infty)$ and $q\in \partial M$, where 
\begin{align*}
C(\varepsilon) & =\frac{1}{2}\int_{\mathbb{R}_{+}^{n}}|\nabla U(y)|^{2}dy-\frac{(n-2)^{2}}{2n-2}\int_{\mathbb{R}^{n-1}}U^{\frac{2n-2}{n-2}}(z,0)dz\\
 & \pm\varepsilon\frac{(n-2)^{3}}{2n-2}\int_{\mathbb{R}^{n-1}}U^{\frac{2n-2}{n-2}}(z,0)dz\mp\varepsilon\frac{(n-2)^{2}}{2n-2}\int_{\mathbb{R}^{n-1}}U^{\frac{2n-2}{n-2}}(z,0)\ln U(z,0)dz\\
 & \mp\varepsilon|\ln\varepsilon|\frac{(n-2)^{3}}{2(2n-2)}\int_{\mathbb{R}^{n-1}}U^{\frac{2n-2}{n-2}}(z,0)dz,
\end{align*}
 and
\[
I_{n-2}^{n-2}=\int_{0}^{\infty}\frac{s^{n-2}}{\left(1+s^{2}\right)^{n-2}}dz
\]
 and $\omega_{n-1}$ is the volume of the $n-1$ dimensional unit
ball.}\\

We compute each term separately. First, we have, by change of variables
and by (\ref{eq:g1}), (\ref{eq:g2}), (\ref{eq:g3}),
\begin{align*}
\int_{M}|\nabla W_{\delta,q}|^{2}d\mu_{g}= 
& \sum_{l,m=1}^{n}\int_{\mathbb{R}_{+}^{n}}g^{lm}
(\delta y)\frac{\partial}{\partial y_{l}}U(y)\frac{\partial}{\partial y_{m}}U(y)\sqrt{g}(\delta y)dy+o(\delta)\\
= & \int_{\mathbb{R}_{+}^{n}}|\nabla U(y)|^{2}dy-\delta(n-1)H(q)\int_{\mathbb{R}_{+}^{n}}y_{n}|\nabla U(y)|^{2}dy\\
 & +2\delta\sum_{i,j=1}^{n-1}\int_{\mathbb{R}_{+}^{n}}y_{n}h_{ij}(q)
 \frac{\partial}{\partial y_{i}}U(y)\frac{\partial}{\partial y_{j}}U(y)dy+o(\delta)
\end{align*}
By simmetry argument we can simplify the last integral to obtain,
in a more compact form
\begin{align*}
\frac{1}{2}\int_{M}|\nabla W_{\delta,q}|^{2}d\mu_{g}=&
\frac{1}{2}\int_{\mathbb{R}_{+}^{n}}|\nabla U|^{2}-\delta\frac{(n-1)H(q)}{2}\int_{\mathbb{R}_{+}^{n}}y_{n}|\nabla U|^{2}\\
&+\delta\sum_{i=1}^{n-1}h_{ii}(q)\int_{\mathbb{R}_{+}^{n}}y_{n}\left(\frac{\partial U}{\partial y_{i}}(y)\right)^{2}+o(\delta).
\end{align*}
Since $\frac{\partial U}{\partial y_{i}}=\frac{\partial U}{\partial y_{l}}$
for all $i,l=1,\dots,n-1$ and by (\ref{eq:integrali3}) we get 
\begin{align*}
\sum_{i=1}^{n-1}h_{ii}(q)\int_{\mathbb{R}_{+}^{n}}y_{n}\left(\frac{\partial U}{\partial y_{i}}(y)\right)^{2}dy & =\frac{1}{n-1}\sum_{i=1}^{n-1}h_{ii}(q)\int_{\mathbb{R}_{+}^{n}}y_{n}\sum_{l=1}^{n-1}\left(\frac{\partial U}{\partial y_{l}}(y)\right)^{2}dy\\
 & =\frac{H(q)}{4}\int_{\mathbb{R}^{n-1}}U^{2}(z,0)dz,
\end{align*}
and in light of (\ref{eq:integrali1}) we conclude that
\[
\frac{1}{2}\int_{M}|\nabla W_{\delta,q}|^{2}d\mu_{g}=\frac{1}{2}\int_{\mathbb{R}_{+}^{n}}|\nabla U|^{2}-\delta\frac{(n-2)H(q)}{4}\int_{\mathbb{R}^{n-1}}U^{2}(z,0)dz+o(\delta).
\]
By change of variables, immediately we obtain 
\[
\frac{1}{2}\int_{M}a(x)|W_{\delta,q}|^{2}d\mu_{g}=\frac{\delta^{2}}{2}\int_{\mathbb{R}_{+}^{n}}a(x)U^{2}(y)\sqrt{g}(\delta y)dy+o(\delta^{2})=O(\delta^{2}).
\]
 Coming to boundary integral, we get, by change of variables, by (\ref{eq:gexp}),
and by expanding $b$, 
\begin{align*}
\frac{n-2}{4}\int_{\partial M}b(z)|W_{\delta,q}|^{2}d\sigma & =\delta\frac{n-2}{4}\int_{\mathbb{R}^{n-1}}b(\delta z)U^{2}(z,0)\sqrt{g}(\delta z)dz+O(\delta^{2})\\
 & =\delta b(q)\frac{n-2}{4}\int_{\mathbb{R}^{n-1}}U^{2}(z,0)dz+O(\delta^{2}).
\end{align*}
By (\ref{eq:Uallaeps}), (\ref{eq:deltaallaeps}) and (\ref{eq:gexp}),
we have 
\begin{align*}
\int_{\partial M}|W_{\delta,q}|^{\frac{2n-2}{n-2}\pm\varepsilon}d\sigma= & \int_{\mathbb{R}^{n-1}}\delta^{\mp\varepsilon\frac{n-2}{2}}U^{\frac{2n-2}{n-2}}(z,0)U^{\pm\varepsilon}(z,0)\sqrt{g}(\delta z)dz+o(\delta)\\
= & \int_{\mathbb{R}^{n-1}}U^{\frac{2n-2}{n-2}}(z,0)dz\pm\varepsilon\int_{\mathbb{R}^{n-1}}U^{\frac{2n-2}{n-2}}(z,0)\ln U(z,0)dz\\
 & \mp\frac{n-2}{2}\varepsilon\ln\delta\int_{\mathbb{R}^{n-1}}U^{\frac{2n-2}{n-2}}(z,0)dz+o(\delta)+O(\varepsilon^{2})+O(\varepsilon^{2}\ln\delta)
\end{align*}
and, since $\frac{(n-2)^{2}}{2n-2\pm\varepsilon(n-2)}=\frac{(n-2)^{2}}{2n-2}\mp\varepsilon\frac{(n-2)^{3}}{2n-2}$
we get 
\begin{multline*}
-\frac{(n-2)^{2}}{2n-2\pm\varepsilon(n-2)}\int_{\partial M}|W_{\delta,q}|^{\frac{2n-2}{n-2}-\varepsilon}d\sigma=-\frac{(n-2)^{2}}{2n-2}\int_{\mathbb{R}^{n-1}}U^{\frac{2n-2}{n-2}}(z,0)dz\\
\pm\varepsilon\frac{(n-2)^{3}}{2n-2}\int_{\mathbb{R}^{n-1}}U^{\frac{2n-2}{n-2}}(z,0)dz\mp\varepsilon\frac{(n-2)^{2}}{2n-2}\int_{\mathbb{R}^{n-1}}U^{\frac{2n-2}{n-2}}(z,0)\ln U(z,0)dz\\
\pm\frac{(n-2)^{3}}{2(2n-2)}\varepsilon\ln\delta\int_{\mathbb{R}^{n-1}}U^{\frac{2n-2}{n-2}}(z,0)dz+o(\delta)+O(\varepsilon^{2})+O(\varepsilon^{2}\ln\delta).
\end{multline*}
Notice that, with the choice $\delta=       d    \varepsilon$ it holds
$o(\delta)+O(\varepsilon^{2})+O(\varepsilon^{2}\ln\delta)=o(\varepsilon)$
and $\varepsilon\ln\delta=\varepsilon\ln       d    -\varepsilon\left|\ln\varepsilon\right|$.
At this point we have 
\begin{align*}
J_{\varepsilon}(W_{\delta,q}) & =C(\varepsilon)+\varepsilon       d    \frac{n-2}{4}\left[b(q)-H(q)\right]\int_{\mathbb{R}^{n-1}}U^{2}(z,0)dz\\
 & \pm\varepsilon\frac{(n-2)^{3}}{2(2n-2)}\ln       d    \int_{\mathbb{R}^{n-1}}U^{\frac{2n-2}{n-2}}(z,0)dz+o(\varepsilon\left|\ln\varepsilon\right|)
\end{align*}
To conclude observe that 
\begin{eqnarray*}
\int_{\mathbb{R}^{n-1}}U^{2}(z,0)dz=\omega_{n-1}I_{n-2}^{n-2} & \text{ and } & \int_{\mathbb{R}^{n-1}}U^{\frac{2n-2}{n-2}}(z,0)dz=\omega_{n-1}I_{n-1}^{n-2}
\end{eqnarray*}
where $I_{\beta}^{\alpha}=\int_{0}^{\infty}\frac{s^{\alpha}}{(1+s^{2})^{\beta}}ds$.
The thesis follows after that we observe that $I_{n-1}^{n-2}=\frac{n-3}{2(n-2)}I_{n-2}^{n-2}$
(for a proof, see \cite{A3}, Lemma 9.4 (b)).
\end{proof}

\subsection{Proof of Theorem \ref{thm:main}: completed}
\begin{proof}
 Let us introduce 
\[
\hat{I}(       d    ,q)=       \alpha_n d    \varphi(q)-\beta_n\ln       d    .
\]
If $q_{0}$ is a local minimizer of $\varphi(q)$ with $\varphi(q_{0})>0$,
set $       d    _{0}={\beta_n\over \alpha_n \varphi(q_{0})}>0$. Thus the pair $(       d    _{0},q_{0})$
is a critical point for $\hat{I}$. Moreover, since there exists a
neighborhood $B$ such that $\varphi(q)>\varphi(q_{0})$ on $\partial B$,
it is possible to find a neighborood $\tilde{B}\subset[a,b]\times\partial M$,
$(       d    _{0},q_{0})\in\tilde{B}$ such that $\hat{I}(       d    ,q)>\hat{I}(       d    _{0},q_{0})$
for $(       d    ,q)\in\partial\tilde{B}$. Since, in the subcritical
case, by (i) of Proposition \ref{expan} we have 
\[
I_{\varepsilon}(       d    ,q)=c_n(\varepsilon)+\varepsilon    \hat{I}(       d    ,q)    +o(\varepsilon)
\]
we get that, for $\varepsilon$ sufficiently small there exists a
$      ( d    ^{*},q^{*})\in\tilde{B}$ such that $W_{\varepsilon       d    ^{*},q^{*}}+\phi_{\varepsilon       d    ^{*},q^{*}}$
is a critical point for $I_{\varepsilon}$. Then, by (i) of Proposition \ref{expan},
$W_{\varepsilon       d    ^{*},q^{*}}+\phi_{\varepsilon       d    ^{*},q^{*}}\in\mathcal{H}$
is a solution for problem (\ref{eq:1.2}) in the subcritical case. 

The proof for the supercritical case follows in a similar way.
\end{proof}
\subsection{Some technicalities}
{\em
If $U$ is a solution of (\ref{eq:PL}) the following equalities hold
\begin{equation}
\int_{\mathbb{R}_{+}^{n}}t|\nabla U|^{2}dzdt=\frac{1}{2}\int_{\mathbb{R}^{n-1}}U^{2}(z,0)dz\label{eq:integrali1}
\end{equation}
\begin{equation}
\int_{\mathbb{R}_{+}^{n}}t|\nabla U|^{2}dzdt=2\int_{\mathbb{R}_{+}^{n}}t|\partial_{t}U|^{2}dzdt\label{eq:integrali2}
\end{equation}
\begin{equation}
\int_{\mathbb{R}_{+}^{n}}t\sum_{i=1}^{n-1}|\partial_{z_{i}}U|^{2}dzdt=\frac{1}{4}\int_{\mathbb{R}^{n-1}}U^{2}(z,0)dz.\label{eq:integrali3}
\end{equation}
}
\begin{proof}
To simplify the notation, we set 
\[
\eta=(z,t)\in\mathbb{R}_{+}^{n},\text{ where }z\in\mathbb{R}^{n-1}\text{ and }t\ge0.
\]
The first estimate can be obtained by integration by parts, and taking into account that 
  $\Delta U=0$; indeed
\begin{align*}
\int_{\mathbb{R}_{+}^{n}}\eta_{n}|\nabla U|^{2}d\eta & =-\sum_{l=1}^{n}\int_{\mathbb{R}_{+}^{n}}U\partial_{l}[\eta_{n}\partial_{l}U]d\eta\\
 & =-\int_{\mathbb{R}_{+}^{n}}U\partial_{n}Ud\eta-\int_{\mathbb{R}_{+}^{n}}\eta_nU\Delta Ud\eta\\
 & =-\frac{1}{2}\int_{\mathbb{R}_{+}^{n}}\partial_{n}\left[U^{2}\right]d\eta=\frac{1}{2}\int_{\mathbb{R}^{n-1}}U^{2}(z,0)dz.
\end{align*}
To obtain (\ref{eq:integrali2}) we proceed in a similar way: since
$\Delta U=0$ we have
\begin{align*}
0 & =-\int_{\mathbb{R}_{+}^{n}}\Delta U\eta_{n}^{2}\partial_{n}Ud\eta=\sum_{l=1}^{n}\int_{\mathbb{R}_{+}^{n}}\partial_{l}U\partial_{l}[\eta_{n}^{2}\partial_{n}U]d\eta\\
 & =\int_{\mathbb{R}_{+}^{n}}2\eta_{n}|\partial_{n}U|^{2}d\eta+\sum_{l=1}^{n}\int_{\mathbb{R}_{+}^{n}}\eta_{n}^{2}\partial_{l}U\partial_{ln}^{2}Ud\eta\\
 & =\int_{\mathbb{R}_{+}^{n}}2\eta_{n}|\partial_{n}U|^{2}d\eta+\frac{1}{2}\int_{\mathbb{R}_{+}^{n}}\eta_{n}^{2}\partial_{n}|\nabla U|^{2}d\eta\\
 & =\int_{\mathbb{R}_{+}^{n}}2\eta_{n}|\partial_{t}U|^{2}d\eta-\int_{\mathbb{R}_{+}^{n}}\eta_{n}|\nabla U|^{2}d\eta
\end{align*}
so (\ref{eq:integrali2}) is proved. Equation (\ref{eq:integrali3})
is a direct consequence of the first two equalities. In fact by(\ref{eq:integrali2})
we have 
\[
\]
\begin{align*}
\int_{\mathbb{R}_{+}^{n}}\eta_{n}|\nabla U|^{2}d\eta & =\int_{\mathbb{R}_{+}^{n}}\eta_{n}\sum_{i=1}^{n-1}|\partial_{i}U|^{2}d\eta+\int_{\mathbb{R}_{+}^{n}}\eta_{n}|\partial_{n}U|^{2}d\eta\\
 & =\int_{\mathbb{R}_{+}^{n}}\eta_{n}\sum_{i=1}^{n-1}|\partial_{i}U|^{2}d\eta+\frac{1}{2}\int_{\mathbb{R}_{+}^{n}}\eta_{n}|\nabla U|^{2}d\eta
\end{align*}
thus 
\[
\int_{\mathbb{R}_{+}^{n}}\eta_{n}\sum_{i=1}^{n-1}|\partial_{i}U|^{2}d\eta=\frac{1}{2}\int_{\mathbb{R}_{+}^{n}}\eta_{n}|\nabla U|^{2}d\eta
\]
 and in light of (\ref{eq:integrali1}) we get the proof.\end{proof}

\end{document}